\documentclass[a4paper,11pt,leqno]{amsart}

\usepackage{graphics}
\usepackage{thmtools}
\usepackage{wasysym}

\usepackage[T1]{fontenc}    

\usepackage{amsthm}
\usepackage{amsbsy,amsmath,amssymb,amscd,amsfonts}
\usepackage[pagebackref=true]{hyperref}
\usepackage{url}            
\usepackage{booktabs}       
\usepackage{nicefrac}       
\usepackage{microtype}      

\usepackage{graphicx,float,latexsym,color}
\usepackage[font={scriptsize,it}]{caption}
\usepackage{subcaption}

\usepackage{makecell}

\usepackage[dvipsnames]{xcolor}

\newtheorem{theorem}{Theorem}

\newtheorem{proposition}{Proposition}

\newtheorem{corollary}{Corollary}
\newtheorem{lemma}{Lemma}
\theoremstyle{remark}

\theoremstyle{definition}
\newtheorem{definition}{Definition}

\hypersetup{
    pdftoolbar=true,        
    pdfmenubar=true,        
    pdffitwindow=false,     
    pdfstartview={FitH},    
    colorlinks=true,       
    linkcolor=OliveGreen,          
    citecolor=blue,        
    filecolor=black,      
    urlcolor=red           
}

\usepackage{lineno}

\usepackage{comment}
\usepackage{microtype}

\begin{document}

\title{When Euler (circle) meets  Poncelet (porism)}


\author[L. G. Gheorghe]{Liliana Gabriela Gheorghe}


\date{October, 2020}

\maketitle













\textbf{Abstract.}

\small{We describe all  triangles that shares the same circumcircle and Euler circle. 
Although this two circles do not form a poristic pair of circles,  we find a poristic circle "in-between"
that enable  to solve this problem using Poncelet porism.
 }
\bigskip

\bigskip

\bigskip

\section{Introduction}

\bigskip

Which triangles  share the same circumcircle and Euler circle? When is a pair of circles the circumcircle and Euler circle of some triangle?
The purpose of this paper is to give a poristic answer to these questions and to 
provide a functorial recipe
to construct them.

In a recent paper, who in fact motivated ours,   P. Pamfilos gave a full answer to these questions in the acute case (see [7] and the references therein); still 
the obtuse angle case was left open.
Other results and related problems were studied by  J. Weaver  in 
[10] and more recently by V. Oxman, in [6].
This problem resemble the first-ever poristic pair of triangles: the i-circle and circumcircle. There are indeed many similarities and our solution consistently use Euler-Chapler formula.

Here, we give a solution that use the inverses of Euler circle, w.r. to circumcircle, as the
"missing poristic hoop" between  circumcircle and the Euler circle. This enabled a functorial construction of  all  triangles that shares the same Euler circle and circumcircle: they will be   the in-touch triangles  associated
with this poristic pair.

Except for a few lines of algebraic computations, that use the versatility of  complex numbers and complex inversion to light the burden, our  solution is geometric and  concise.
Nevertheless, a careful reader would realize that this brevity strongly relies on  Poncelet porism and on the related  Euler-Chapler  formulas.

\section{A poristic tern}

   \begin{definition}
    
 We say that $(\mathcal{C}_1,\mathcal{C}_2)$ form a pair of poristic circles (for $n=3$) if there exists a some triangle $\triangle{A_1B_1C_1},$
whose vertices are on $\mathcal{C}_1$ and whose sides (or their lines) tangents circle $\mathcal{C}_2.$
We shall call 
$\mathcal{C}_1$ the inner circle and $\mathcal{C}_2$ the outer circle.

  If $A_2,B_2,C_2$  are the tangency points of the sides of $\triangle{A_1B_1C_1}$
with circle $\mathcal{C}_2$,
we say that the triangles $\triangle{A_1B_1C_1}$ and $\triangle{A_2B_2C_2}$
are a compatible pair of poristic triangles w.r. to this couple,
and call  $\triangle{A_1B_1C_1}$ is the circumscribed (or outer) triangle and 
$\triangle{A_2B_2C_2}$ is the in-touch (or inner) triangle. 
\label{definition:poristic-circles}
\end{definition}

--------------------------------------

\textbf{Keywords and phrases: } Euler circle, Euler-Chapler formula, Poncelet porism, negative pedal curve, dual curve, inversion. 

\textbf{(2020)Mathematics Subject Classification: }51M04, 51M15, 51A05.

These circles are either circumcircle and inscribed circle, or  circumcircle and some exinscribed circle; 
these  circles may   be secant, as well.

Note that this is an  ordered  pair;
we shall call the first circle, 
the "outer" and the second circle, the "inner."

Poncelet porism 
still works when the two circles
(or conics) intercepts; see [8]. Therefore,  
for any choice
of an initial point $A_1$ on the 
outer circle $\mathcal{C}_1,$ 
either  no  tangents  $t_b$ and $t_c$ from $A_1$ to $\mathcal{C}_2$  intercept again 
$\mathcal{C}_1$ or, the two  intercept 
 $\mathcal{C}_1$ in $B_1,C_1,$ and then
the line $B_1C_1$ tangents $\mathcal{C}_2,$ as well. 

Unfortunately,
the circumcircle $\mathcal{C}$ and Euler circle
$\mathcal{E}$ of a triangle do not match poristically; nevertheless, there is 
poristic hoop that join  them: 
$\mathcal{E}'$, the inverses of Euler circle w.r. to circumcircle.

\begin{figure}[H]
    \centering
    \includegraphics[trim=400 20 400 20,clip,width=.7\textwidth]{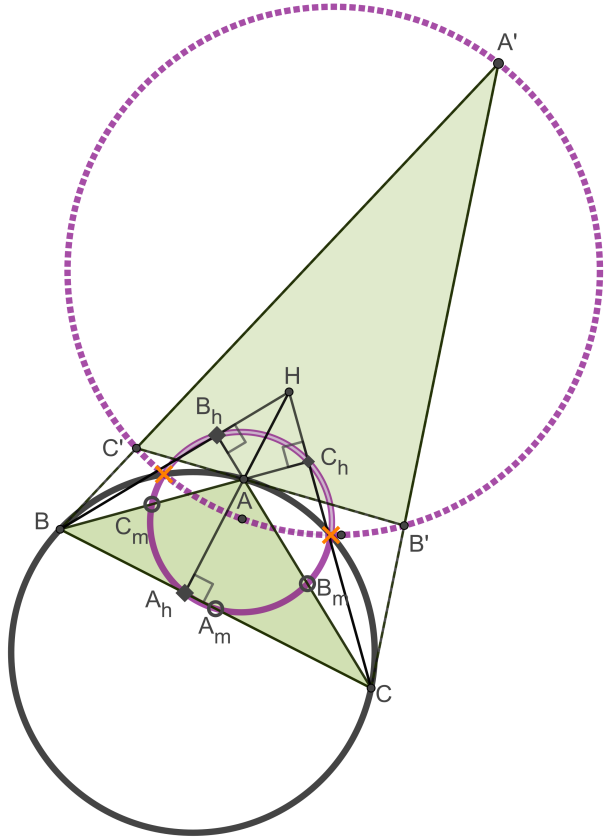}
    \caption{The Euler circle $\mathcal{E}$ (bordeaux) is the circumcircle of the midpoints $A_m,B_m,C_m$ of the sides of  $\triangle ABC.$
    The inverses of the midpoints w.r. to circumcircle, 
    are the poles of its sides and 
    located on the inverse of Euler circle $\mathcal{E}$ (bordeaux, dotted). 
    The sides of $\triangle{A'B'C'}$  tangents
     in $A,B,C$
    the circumcircle of the triangle 
     $\mathcal{C}$. 
     }
    \label{fig:0801}
\end{figure}

Refer to Fig.\ref{fig:0801}.

\begin{lemma} Let $\triangle {ABC}$ and let $\mathcal{C}$ its circumcircle.

i) The tangents 
in $A,B,C$ at $\mathcal{C}$ determine $\triangle{A'B'C'}.$ 
Then  $A',B',C'$ are the precisely the inverses of the midpoints $A_m,B_m,C_m$
 of the sides of the triangle w.r. to $\mathcal{C}.$

ii) Reciprocally, if $A',B',C'$ are the inverses
of the midpoints $A_m,B_m,C_m$, of the sides $BC,$
$CA$, $AB$ of $\triangle{ABC}$
then the lines $A'B'$,
$A'C'$, $B'C'$ tangents the circle in $C,B,A,$ respectively.

\label{lemma:polar-polo-simetrico}
\end{lemma}

\begin{proof} i) By construction, 
 $\triangle{A'BO}$ and $\triangle{A'CO}$ are two congruent right-angled triangles.
 Since $\triangle{BOC}$ is isosceles, its median $OA_m$ is also an altitude.
 Thus, 
$BA_m$  is  altitude in $\triangle{A'BO}$, hence $OA_m\cdot OA'=OB^2,$ and $O,A_m,A'$ are collinear, proving that
$A'$ is the inverses of $A_m$ w.r. to circle $\mathcal{C}$.

ii) If $A'$ is the inverse of $A_m,$ the midpoint of $BC,$ than, by the construction of the inverses, 
$A'B\perp OB;$  
similarly, $C'$ is the inverse of $C_m,$ the midpoint of $BC,$ than $C'B\perp OB;$
hence $A',B,C'$ are collinear and $A'C'$ tangents the circle in $B.$ Similarly  $B'C'$,
and $A'B'$ tangents the circle in $A$ and $C,$ respectively.
\end{proof}

By construction, $A',B',C'$ are the poles of the sides of $\triangle{ABC};$ therefore, Lemma ~\ref{lemma:polar-polo-simetrico} reads as follows:

\begin{corollary}
The circumcircle of $\triangle{A'B'C'}$
 is the inverse of Euler circle $\mathcal{E}$ of $\triangle{ABC}$ w.r. to $\mathcal{C}.$
 \label{cor:corolario1}
\end{corollary}
Lemma \ref{lemma:polar-polo-simetrico} and the simple fact that inversion in circle is an involution led to  the following.

 \begin{theorem} Let $\triangle {ABC}$
and let  $\mathcal{C}$ its circumcircle. The tangents 
in $A,B,C$ at $\mathcal{C}$ determine  $\triangle{A'B'C'};$ let $\mathcal{E}'$ be its circumcircle.
Then:
\begin{enumerate}
\item
$(\mathcal{E}', \mathcal C)$ is a poristic pair (for $n=3$);
\item
$\triangle{A_1B_1C_1}$ have circumcircle $\mathcal{C}$ and Euler circle $\mathcal{E}$ if and only if it is an in-touch triangle  w.r. to  
$(\mathcal{E}', \mathcal C)$.
\end{enumerate}
\label{proposition:simetric-point-pole}
\end{theorem}

\begin{proof} 
i) is a direct consequence of the construction and of Poncelet porism.

ii) 
$\Rightarrow$ is also a direct consequence of the construction and of Lemma \ref{lemma:polar-polo-simetrico}.

$\Leftarrow$ 
We need to prove that if some triangle $\triangle{A_1B_1C_1}$ is in-touch w.r. to $(\mathcal{E}', \mathcal C)$, then its circumcircle is $\mathcal{C}$ and its Euler circle is $\mathcal{E}$.

Refer to figure~\ref{fig:0801} and ~\ref{fig:0821}; to avoid desnecessary repetition of this figure,   points $A_1,B_1,C_1$ reads $A,B,C$ and so on.

By hypothesis,  $\triangle{A_1B_1C_1}$ is in-touch w.r. to $(\mathcal{E}', \mathcal C)$;
so there exists a point  $A_1'\in \mathcal{E}'$ such that if   $A_1'B_1',$
$A_1'C_1'$ the two tangents from $A_1'$ to circle $\mathcal{C};$ then, by Poncelet porism, $B_1'C_1'$ also tangents $\mathcal{C};$ and tangency point of
 these tangents with circle 
$\mathcal{E}'$ determines the in-touch triangle.

We only have to prove that Euler triangle of $\triangle{A_1B_1C_1}$ is $\mathcal{E}$.
By construction, $A_1'B_1$ and $A_1'C_1$ are two tangents to circle $\mathcal{C};$ therefore $\triangle{A_1'B_1C_1}$ is isosceles and if we denote with $A_{1m}$ the midpoint of segment $[B_1C_1],$ then  $A'_1A_{1m}$ is perpendicular on $B_1C_1;$
on the other side, $OA_{1m}$ is perpendicular on the cord
$B_1C_1,$ since
$B_1,C_1$ are on the circle $\mathcal{C}$ and 
$O$ is its center. Therefore, $A_1',A_{1m}$ and $O$ are collinear.
Hence, in $\triangle {OB_1A_1'}$ which is right-angled at $B_1,$ as line $A_1'B_1$ is a tangent to $\mathcal{C}$, 
$$OB_1^2=OA_{1m}\cdot OA_1'$$ which means that point $A_{1m}$
and $A_1'$ are inverses w.r. to circle $\mathcal{C}.$
Similarly, $B_{1m}$ and $B_1'$, and $C_{1m}$ and $C_1'$.
Thus, the midpoints of the intouch triangle $\triangle{A_1B_1C_1}$ are the inverses of $A_1',B_1',C_1',$
w.r. to $\mathcal{C}$.
These midpoints determine the Euler circle, say $\mathcal{E_1}$ of the in-touch triangle $\triangle{A_1B_1C_1};$ but since they proved to be the inverses of three points that belong to  $\mathcal{E}'$ they are on the inverse of $\mathcal{E}'$ w.r. to $\mathcal{C},$ hence on $\mathcal{E}.$

\end{proof}

\begin{figure}[H]
    \centering
    \includegraphics[trim=300 40 300 0,clip,width=.7\textwidth]{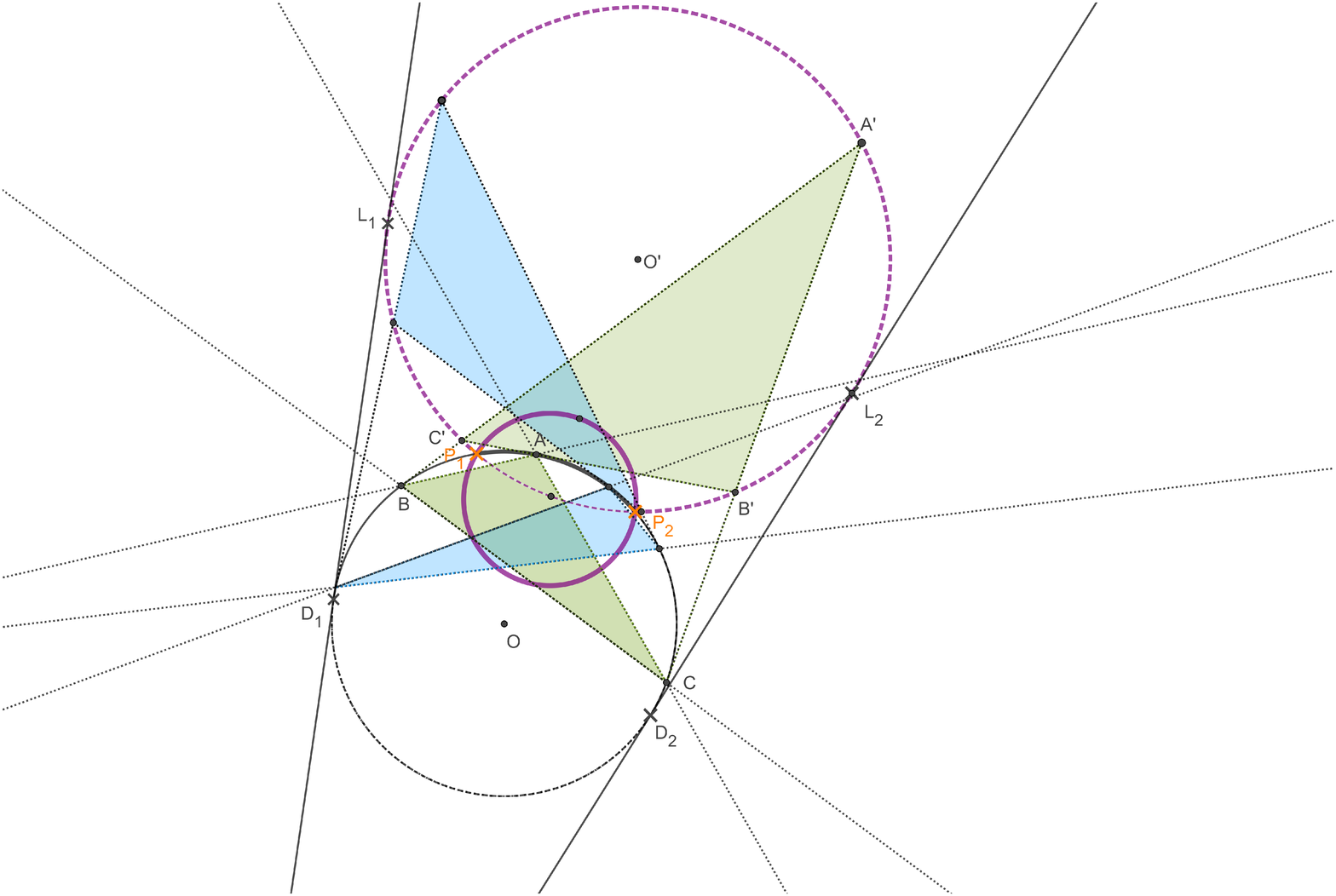}
    \caption{
    If $\triangle{ABC}$ is obtuse, then its Euler circle (bordeaux) intercept its circumcircle in $P_1,P_2;$ then
    the obtuse vertice lies necessarily on arc $P_1P_2$
    and one side $B'C'$ of its outer triangle tangents the circumcircle precisely at $A;$ if $L_1,L_2$ and $D_1,D_2$ are the interception of the common tangents, at the two circles, then 
    $B',C'$ lie,  on arcs $P_1L_1$ and $P_2L_2$, respectively. $A'$ lie on arc $L_1L_2.$  
    }
    \label{fig:0821}
\end{figure}

\subsection{The obtuse case.}
Let us  take a closer look to the obtuse triangles.
In this case, some arcs of circumcircle $\mathcal{C},$ or on 
$\mathcal{E}',$
are either "unreachable"
or "infertile:" either they cannot be reached by any tangent line or there is not possible to draw 
tangents to $\mathcal{C}$ form that points of   of $\mathcal{E}'$.

Referring to Fig~\ref{fig:0821}
\begin{lemma}
    The Euler circle $\mathcal{E}$ and  the  circumcircle $\mathcal{C}$
    of a triangle are secant if and only if the triangle is obtuse.
    
    If this is the case, let the  obtuse angle be in $B$ and let
    $P_1$, $P_2$  the intersection points of Euler circle and circumcircle.
    Then   $B$ locates inside the Euler circle, on the arc of $\mathcal{C}$ 
    delimited by $P_1$ and $P_2$.
 \label{lemma:euler-obtuse}   
\end{lemma}
\begin{proof} Refer to figures ~\ref{fig:0801} and  ~\ref{fig:0821}.
Let $A_h$ and $C_h$ the foots of the altitudes
from $A$ and $C$ to the sides of triangle. 
Since $\angle{B}$ is obtuse, these two points $A_h$ and $C_h$ 
locates outside the sides $[BC]$ and $[AC].$
Thus, $C_m C_h>AC_m$ and $A_m A_h>CA_m.$
Therefore, $B$ is located on  the segments $[A_m C_h]$ 
and on $[C_m A_h];$ since  
$A_m,A_h,B_m,B_h$ are on the Euler  circle, by convexity,
these segments thus, point $B$ itself is inside the Euler circle.

Since $B$ is on circle $\mathcal{C},$ it therefore located  on  the arc $P_1P_2$ of $\mathcal{C},$ 
situated inside the Euler circle $\mathcal{E}.$
\end{proof}

\begin{corollary}
With notations as in Lemma \ref{lemma:euler-obtuse}, the triangles sharing the same circumcircle and Euler circle with an obtuse $\triangle{ABC}$ have necessarily one vertex  on arc $P_1P_2.$
\end{corollary}
In the obtuse case, the circles $(\mathcal{E}',\mathcal{C})$ form a porisic pair of secant circles. 
The intersection of tangents in $P_1$ and $P_2$ to the circumcircle $\mathcal{C}$ 
with $\mathcal{E}'$ give explicit  restriction on "fertile" or "sterile" arcs: the arcs of  $\mathcal{E}'$ 
that may or may not contain a vertex of some triangle,  that generate a triangle inscribed in $\mathcal{C}.$

\begin{figure}[H]
    \centering
    \includegraphics[trim=100 80 100 10,clip,width=.7\textwidth]{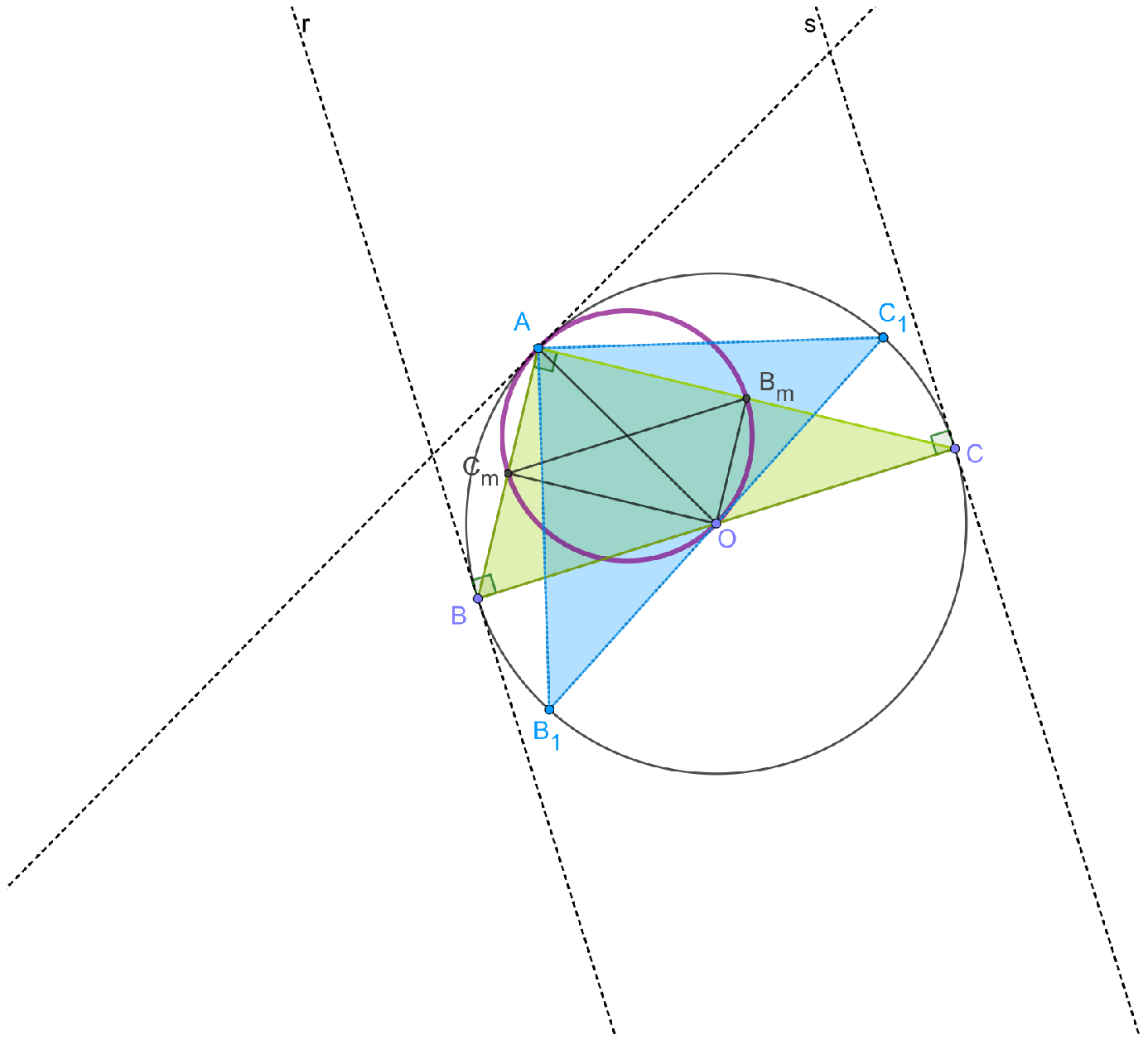}
    \caption{
    When the triangle is right-angled its Euler circle tangents its circumcircle
    }
    \label{fig:euler-retangulo}
\end{figure}

\subsection{The right-angle case.}

Euler circle tangents internally the circumcircle iff the triangle is right-angled; in this case, all triangles sharing the circumcircle and Euler circle have a common point in vertices in $A$, the tangency point; in this case, the polars of the vertices does not define any circle and the inverses of Euler circle w.r. to circumcircle is a line: the tangent in $A.$

\section{Pairs of compatible circumcircle-Euler circle}

Till now, our  pair of circles
 $(\mathcal{C},\mathcal{E})$, were respectively, the 
circumcircle and Euler circle of some
pre-existent triangle.

Euler-Chapler formulas give the necessary and sufficient conditions in order that two circles with known distance between their centers and known radius 
   be  either the circumcircle and the incircle 
or the circumcircle and  ex-inscribed circle of some triangle.
Here, we formulate  similar conditions for the (non poristic) pair circumcircle and Euler circle. 

First, some necessary conditions.
 
\begin{theorem} Let $\mathcal{C(O,R)}$ and $\mathcal{E}(N,r)$ the circumcircle and Euler circle of some triangle $\triangle{ABC}.$
Then $r=\frac{R}{2}.$ Further,

i) if  $\triangle{ABC}$ is acute, then $ON<\frac{R}{2}.$
 
     ii)  if $\triangle{ABC}$ is right-angled, then $ON=\frac{R}{2}.$
     
     iii) if $\triangle{ABC}$ is  obtuse then $ON>\frac{R}{2}.$
    \label{thm:necessary-Euler-chapler}
\end{theorem}

\begin{proof} Refer to 
figure ~\ref{fig:0801}.
 Euler circle is the circumcircle of $\triangle A_mB_mC_m$ 
which similar to $\triangle{ABC}$ itself, hence
the radius of their circumcircle shares the same proportionality rate.
Hence $r=\frac{R}{2}.$

i)     If the triangle   is acute,
its  orthocenter $H$  is 
inside the triangle; thus, $OH<R.$
On the other hand 
the center of Euler circle, $N,$ is the midpoint of $OH.$
This implies  $ON<\frac{R}{2}.$
Since the radius of the Euler circle is $\frac{R}{2},$ 
this ensures that Euler circle is contained into the circumcircle.

ii) If the triangle is right-angled, its orthocenter coincides 
with its right-angle  vertex, 
and the Euler circle of a right-angled triangle
tangents internally the circumcircle
and pass through its center.

iii) If the triangle is obtuse,  at least one feet of its altitudes 
lie outside the triangle, and outside the circumcircle.
Since the midpoints 
$A_m,B_m,C_m$  are on Euler circle, then by convexity, 
 are contained into the circumcircle  $\mathcal{C}$.
 Euler circle of an obtuse triangle therefore contain  both 
points inside and outside  the circumcircle, hence in this case the two are  secant.
\end{proof}

Now we shall prove that these  conditions are also sufficient.
 
 \begin{theorem} Let $\mathcal{C}=\mathcal{C(O,R)}$ and $\mathcal{E}=\mathcal{E}(N,\frac{R}{2})$ be two circles such that
 $0\leq ON<\frac{3}{2}R.$ Then there exists infinitely many triangles whose circumcircle is 
 $\mathcal{C(O,R)}$ and whose Euler circle is  $\mathcal{E}(N,\frac{R}{2}).$
 Further, 
 
i) if $ON<\frac{R}{2},$ then all these triangles  are acute; 
 
     ii) if  $ON=\frac{R}{2},$ then all triangles are right-angled;
     
     iii) $ON>\frac{R}{2},$ then all triangles are obtuse.
    \label{thm:sufficient-Euler-chapler}
\end{theorem}
  We shall give an indirect prove:  
 we show that if these conditions are fulfilled, then   
 $(\mathcal{E}', \mathcal{C})$ form a poristic pair of circles.
 To see this, we shall need the following
 fact.
 
 \begin{lemma}
If  $(\mathcal{C}, \mathcal{E})$ verify the conditions above, then 
     $(\mathcal{E}', \mathcal{C})$ form a poristic pair for $n=3.$
 \end{lemma}  
   
   \begin{proof} 
   Two circles form a poristic pair for $n=3$ if and only if they are either the circumcircle and inscribed circle or circumcircle and exinscribed circle. 
   Euler-Chapler  formulas guarantees that this hapens if and only if  $\mathcal{E}'$ and $\mathcal{C}$  verify 
     one of the following relations: 
    $$i ) (R_1-r_1)^2=d_1^2+r_1^2;\;\;\;(1)$$ 
    $$ ii) (R_1+r_{1ex})^2=d_{1}^2+r_{1ex}^2  \;\;\;(2)$$
    
    The first relation 
    is the necessary and sufficient condition in order to a pair  of circles which have
     radius $R_1,r_1$ and whose centers dist $d_1,$
    be, respectively, the circumcircle and the
    inscribed circle of some triangle.
    
    The second relation 
    is the necessary and sufficient condition in order to a pair of circles
    of radius $R_1,r_{1ex}$ and whose centers dist $d_1,$
    be, respectively, the circumcircle and the
    ex-inscribed circle of some triangle.

   For our proof, $R_1$ is  the radius of   $\mathcal{E}'$, (as the  external circle), 
    $r_1$ the radius of $\mathcal{C}$ (as the "inscribed circle") and $r_{1ex}$ the radius 
    of an "ex-inscribed" circle; and  
     $d_1$ is  the distance between the centers of $\mathcal{E}'$ and $\mathcal{C}$.

    We prove that the first case 
   occurs when $0\leq ON<\frac{R}{2}$ and the second, when $\frac{R}{2}< ON<\frac{3}{2}R.$

The  proof now reduces to  a straightforward verification.
We use complex numbers.

In order to simplify the computations, assume, 
without lose of generality, that
 $\mathcal{C}$ is the unit circle, (the circle centred in $0$ and radius $1$ and let $d$ be the center of a circle $\mathcal{E}_d$ of radius 
$\frac{1}{2}$ (half if the radius of $\mathcal{C}$) and  whose interior  intercepts $\mathcal{C}.$

Again, without lose of  generality, we may assume that its center is a point  on the real positive line; therefore
$d\in[0,3/2)$ and $d\neq 1.$ 
We shall treat the case $d=1$ separately, since it is the tangency case. 

In order to prove the first assertion, note that the center of the given circle
$\mathcal{E}=\mathcal{E}_d$
whose radius is half the radius of  $\mathcal{C},$ hence $\frac{1}{2},$ is located at a distance $d\in[0, \frac{1}{2})$, being internal circle.
We want to prove that $(\mathcal{E}'_d,\mathcal{C})$ form a poristic pair
of circles for $n=3.$

The diameter of $\mathcal{E}_d$ is $[AB]$ where 
$$A=d-\frac{1}{2} \;\hbox{and}\; B=d+\frac{1}{2};$$ let $d_0=2d\in(0,1);$
hence $A=\frac{1}{2}(d_0-1)<0,$ $B=\frac{1}{2}(d_0+1)>0.$
Now let $$A'=\frac{2}{d_0-1}\; \hbox{and}\; B'=\frac{2}{d_0+1}$$ be the inverses of $A,B$
w.r. to the unit circle; then the inverses of $\mathcal{C}_d$ is the circle
$\mathcal{C}'_d$ whose diameter is $[A'B'].$
Its center is 
$$O'=\frac{A'+B'}{2}=\frac{2d_0}{d_0^2-1}\;\hbox{and its radius is} R'=\frac{2}{d_0^2-1}.$$
Now we check that these two circles verifies the Euler-Chapler relations.
$$(R-r)^2=d^2+r^2,\hbox{where}\; r=1,\hbox{and}\; R=\frac{2}{d_0^2-1}, 
d=\frac{2d_1}{1-d_1^2}.$$
This relation becomes
$$(\frac{2}{d_0^2-1}-1)^2=\frac{4d_0^2}{(d_0^2-1)^2}+1,$$  
$$\big(\frac{d_0^2+1}{d_0^2-1}\big)^2=\frac{4d_0^2+d_0^4-2d_0^2+1}{(d_0^2-1)^2}.$$
which is obviously verified!

The second assertion, in which the center of the given circle
$\mathcal{E}_d$
with radius 
$\frac{1}{2}$ is at a distance $d\in[\frac{1}{2},\frac{3}{2})$, hence is a secant circle, proves just in the same way.

The diameter of $\mathcal{C}_d$ is $[AB],$ where 
$A=d-\frac{1}{2}$ and $B=d+\frac{1}{2};$ let $d_0=2d\in(1,3);$
hence $A=\frac{1}{2}(d_0-1),$ $B=\frac{1}{2}(d_0+1).$
Now let $A'=\frac{2}{d_0-1}$ and $B'=\frac{2}{d_0+1}$ be the inverses of $A,B$ w.r. to the unit circle; then the inverses of $\mathcal{E}_d$ is the circle
$\mathcal{E}'_d$ whose diameter is $[A'B'].$
Its center is $O'=\frac{A'+B'}{2}=\frac{2d_0}{d_0^2-1}$ and its radius is $R'=\frac{2}{d_0^2-1}.$
Now we check that these two circles verifies the relation given by Euler-Chapler theorem.
$$(R+r_{ex})^2=d^2+r^2_{ex},\;\;\;r_{ex}=1,\;\; R=R'.$$
This relation become
$$(\frac{2}{d_0^2-1}+1)^2=\frac{4d_0^2}{(d_0^2-1)^2}+1,$$  
$$\big(\frac{d_0^2+1}{d_0^2-1}\big)^2=\frac{4d_0^2+d_0^4-2d_0^2+1}{(d_0^2-1)^2},$$
which is obviously verified.

This expression does not make sense when $d=1.$ But this only occurs when the Euler circle tangent the circumcircle.
The tangency case is straightforward, it does not require any  poristic computations and we omit it.

    \end{proof}
    
We thus proved that any circle whose interior  intercepts a given circle 
and whose radius is half that of the circle,  is an Euler circle for some (hence infinitely many) triangles.
 
Finally, let us illustrate how we can generate all the triangles that shares the same  circumcircle and Euler circle with a given triangle and how we can draw the triangle with a prescribed vertices.

Any poristic pair of circles 
generates two  poristically-related families of  triangles: 
the circumscribed triangles, whose vertices are on the outer circle
and the inner (or in-touch) triangle, whose vertices are tangency point of
tangents form points of  outer circle,  to the inner circle. Thus, any outer triangle is  poristically bounded to one 
(and only one) inner triangle.

There are two families of poristic triangles w.r. to $(\mathcal{E}', \mathcal{C}):$
$\mathcal{T'},$ containing  all triangles 
 whose vertices are on $\mathcal{E}'$ and whose sides tangent the circle 
$\mathcal{C}$ (the "outer" triangles) and $\mathcal{T}$ the family of 
all triangles whose vertices are the tangency point of the sides of triangles in $\mathcal{T'},$
with circle $\mathcal{C},$ (the inner triangles, whose vertices are on 
$\mathcal{C}$ and whose Euler circle is $\mathcal{E}.$

Note that Poncelet porism is still valid when the poristic circles
$(\mathcal{E}',  \mathcal{C})$ are secant; the only restriction is on the location of the initial point $B_1':$ if it  lie inside the circumcircle, we cannot trace tangents from it, to $\mathcal{C}$. Equivalently, one side of the poristic in-touch triangle (inscribed in $\mathcal{C}$)
touch the arc $\mathcal{C})$ on a point inside $\mathcal{E}'$.  

The two circles    of the poristic pair are disjoint, if the triangle is acute and secant, if 
the triangle is obtuse.

A geometric  recipe to construct some  triangle, 
given its circumcircle, its Euler circle, and one of its vertices is now straightforward.
\begin{corollary} Let  $\mathcal{C}$ and $\mathcal{E}$ two non-tangent
circles that are circumcircle and Euler circle for some triangle $\triangle{ABC}$ and let
and $A_1$ a point on $\mathcal{C}.$

 i) Let $\mathcal{E}'$ be the inverses of $\mathcal{E}$ w.r. to 
$\mathcal{C}.$ Let $a'$, the tangent at $A_1$ to the circle  $\mathcal{C};$
then:
\begin{itemize}
    \item 
i) if $a'$ does not intercept the circle $\mathcal{E}'$ in two distinct points 
 there is no such triangle.
\item
ii) if $a'$ is secant to $\mathcal{E}'$
in $B'$ and $C'$,
then let the tangents $b',c'$ from $B'$ and $C'$ 
tangent the circle $\mathcal{C}$ in two other points $B,C;$ these
who are the vertices of the required triangle.
\end{itemize}
\end{corollary}
 
 Note that the case $i)$ can only occur when the two circles are secant. In this case, this happens
 because the point $A_1$ lie into an unreachable arc of $\mathcal{C}.$

\begin{corollary}
A pair $(\mathcal{C},\mathcal{E})$ of 
internally tangent circles at a point $A$
are the circumcircle and Euler circle of some triangle, 
if and only if the circle $\mathcal{E}$ pass through the center of $\mathcal{C}.$ 
The triangles which have a vertex in $A$ and whose side is a diameter of the circle passing through point $A_1,$
are inscribed in $\mathcal{C}$ and share the Euler circle 
$\mathcal{C}$ and is the requested triangle.
If the two circles  are externally tangent, there is no such a  triangle.
\label{corollary:poristic-euler-pair}

\end{corollary}

\begin{figure}
\centering
\includegraphics[trim=300 150 300 60,clip,width=1.0\textwidth]{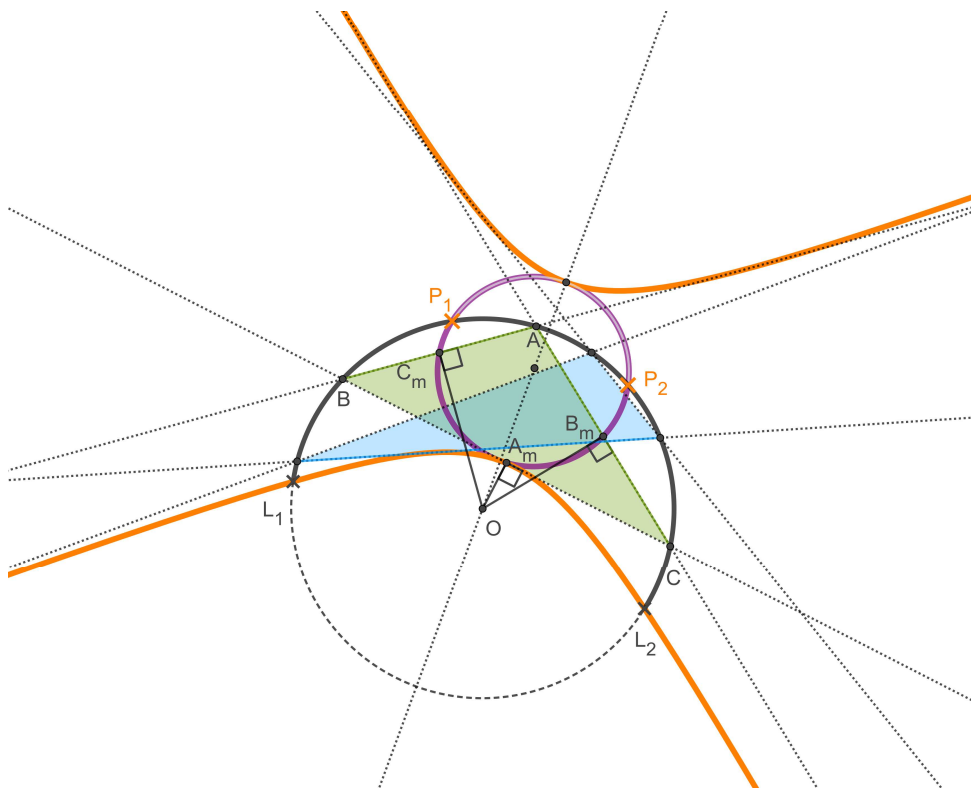}
\caption{
The i-conic and the circumcircle form a poristic pair of conics.
The i-conic (orange hyperbola) is precisely the negative pedal curve of Euler circle, w.r. to pedal point $O,$ the circumcenter of $\triangle ABC.$ 
All poristic triangles 
w.r. to circumcircle and the i-conic share the same Euler circle.}
\label{fig:1000_Euler_pedal_Passo0}
\end{figure}

\section{A PORISTIC PROOF}

This section reveals yet another  poristical bound between Euler circle and  circumcircle.
This time, we shall deal with two  poristic pairs: the circumcircle  and the i-conic, the conic  inscribed into the triangle and whose focus is in its  circumcenter.
The  relationship between the i-conic, circumcircle and  Euler circle, in a twisting of interchangeable poristic terns,  spotted in 
a recent paper by R. Garsia and D. Reznik (see [2]), will enable  a very short 
 poristic proof of our main result.

To this end, 
 we shall use inversive techniques and polar reciprocity; the reader not familiar  with this concepts may want to see the beautiful books 
[1], [9]; for ad-hoc  details see [4], Appendix.

 Let $\triangle{ABC}$ and  
    $\mathcal {C}$ its circumcircle;  let $\mathcal{E}$ its Euler circle and 
    $\mathcal{E}'$ 
    the inverse of Euler circle, w.r. to the circumcircle.
    
    We shall denote by
     $\Gamma$ its i-conic, the conic tangent to the sides of triangle $\triangle{ABC}$ and with focus in $O,$ the circumcenter.

The following lemma will prove that the poristic family of triangles associated with this pair, share the same Euler circle.

\begin{definition} (see Appendix)
The negative pedal of a curve $\gamma$ w.r. to a pedal point $D$ is the curve $\Gamma,$ whose tangents are the perpendicular on $M$ to $DM,$ as $M$ sweeps the curve $\gamma.$
\label{def:pedal_negativa}
\end{definition}
The following fact is classic.
\begin{lemma} 
The negative pedal curve  $\Gamma=\mathcal{N}(\mathcal{E})$ of a circle   $\mathcal{E},$ w.r. to a pedal point $D$  
is a conic $\gamma_D,$ which has a focus in $D,$ and  whose  axis coincides with the diameter  of  $\mathcal{E}$ through $D$.

 $\Gamma$ is an ellipse (resp. hyperbola), iff $D$ is inside (resp. outside) $\mathcal{E}$.
\label{lemma:negative-pedal-iconic}
\end{lemma}

The  reader not acquainted with these concepts,  
may either see  [5], and references therein, or Appendix of
[4], for a short briefing.

This fact has an immediate consequences; see also [2], Theorem 1.
\begin{corollary}
The i-conic ${\gamma}_{D}$ is precisely 
   the negative pedal curve of  ${\mathcal{E}}_D,$ w.r. to pedal point $D.$

The in-ellipse and Euler pedal circle  tangents at the vertices of i-conic,
have the same center and  the main axis of the in-ellipse is a diameter of the Euler pedal circle.
\label{lemma:euler-negative-pedal}
\end{corollary}

\begin{figure}
\centering
\includegraphics[trim=300 20 300 60,clip,width=1.0\textwidth]{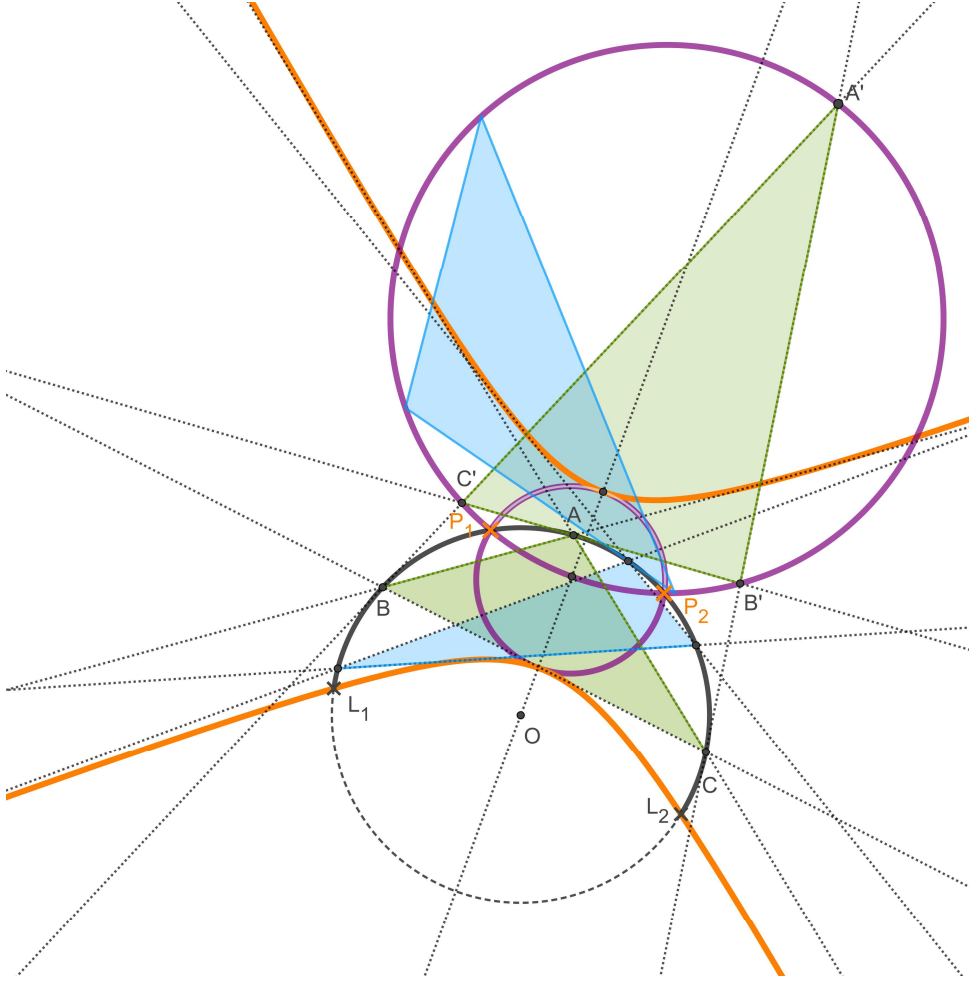}
\caption{The polars of the vertices of 
$\triangle{ABC}$ poristic w.r.$(\mathcal{C},\Gamma)$
determines  $\triangle {A'B'C'},$ who is poristic w.r. to 
$(\mathcal{E}',\mathcal{C}),$ where $\mathcal{E}'$
(bordeaux, doted)
is the inverses of Euler circle.
Thus, circles $(\mathcal{E}',\mathcal{C}),$ being the circumcircle and the excircle of $\triangle{A'B'C'}$ form a poristic pair and $\triangle{ABC}$ is intouch w.r. to this pair.
The family of intouch contains all triangles that
shares the circumcircle $\mathcal{C}$ and the Euler circle $\mathcal{E}.$
}
\label{fig:1010_Euler_pedal_Passo2}
\end{figure}

\begin{lemma} Let  $\triangle{ABC}$ and $O$ its circumcenter.
    The i-conic $\Gamma$ is the negative pedal of Euler circle $\mathcal{E},$
    w.r. to pedal point $O.$
     \label{lemma:i-conica-pedal-negativa}
\end{lemma}

\begin{proof} Refer to figure
~\ref{fig:1000_Euler_pedal_Passo0}.

Let $C_m$ be he midpoint of $AB;$ then $OC_m\perp AB;$
since $C_m$ is on
 $\mathcal{E},$
 then, by the definition of a negative pedal curve, the line $AB$  tangents $\Gamma;$ the same for the other two sides. Therefore, the negative pedal curve of Euler circle is precisely the i-conic of $\triangle{ABC}.$
 \end{proof}

\begin{corollary}
The triangles that share the same circumcircle and Euler circle, are precisely those poristically inscribed into $\mathcal{C}$
and circumscribed to ${\Gamma}$.

\end{corollary}

Now we shall formulate the result in terms of a pair of circles and provide a poristic proof of our main result.

\begin{theorem} (see Pamfilos [7], Theorem 5)
The triangles that shares the same Euler circle $\mathcal{E}$
    and circumcircle $\mathcal{C}$ are  those whose vertices are tangency points (with circle $\mathcal{C}$)  of  the sides of triangles inscribed into $\mathcal{E'}$
and circumscribed to  $\mathcal{C}$. 
\end{theorem}
\begin{proof}  Refer to figure 
\ref{fig:1010_Euler_pedal_Passo2}

 Perform a dual transform w.r. to  $\mathcal{C}.$
 The vertices $A,B,C$ converts into their polars, which are the tangents $a',b',c'$
 that determines $\triangle{A'B'C'}$. By this polarity,
  the circumcircle  $\mathcal{C}$ does not change,  while the i-conic $\Gamma,$ whose focus is  $O,$ the inversion center, converts into the circumcircle of the polar triangle
 $\triangle{A'B'C'}.$ 
   By 
  ~\ref{cor:corolario1} this is 
the inverse  of Euler circle, $\mathcal{E'}.$ 
    Therefore,   $\triangle{ABC}$ is inscribed in $\mathcal{C}$ 
 and circumscribed to the i-conic $\Gamma$ if and only if
  $\triangle{A'B'C'}$ is  inscribed into
  the symmetric of Euler circle w.r. to the circumcircle
  and circumscribed to 
  $\mathcal{C}.$ 
  
  We thus  poristically generate two families of triangles, one circumscribed, other in-touch, the later being  precisely the family of all triangles sharing the same Euler circle and circumcircle. 
  \end{proof}
  
 The obtuse-angle case does not need a special treat as   very few  change. The i-conic will be a hyperbola and the poristic triangles are again the in-touch triangles with respect to the dual system.
 Some arc of the circumcircle $\mathcal{C}$ 
 is not reachable. Nevertheless, this does not alter  the result, nor  the construction. The description of the unreachable arcs is similar to those  we provide in the previous section, so we omit it.

 Despite  circumcircle and  Euler circle are not  poristic pair, 
 a providential  $i-conic$ and its negative pedal curve put the things into the right perspective and enable to give
 a poristic solution to a non-poristic problem.

\bigskip
\bigskip
\bigskip

DEPARTAMENTO DE MATEMÁTICA

UNIVERSIDADE FEDERAL DE PERNAMBUCO

RECIFE, (PE) BRASIL

\textit{E-mail address}:

\texttt{liliana@dmat.ufpe.br}
\bigskip
\bigskip
\end{document}